\documentclass{article}

\usepackage[final,nonatbib]{nips_2017}

\usepackage[utf8]{inputenc} 
\usepackage[T1]{fontenc}    
\usepackage{hyperref}       
\usepackage{url}            
\usepackage{booktabs}       
\usepackage{amsfonts}       
\usepackage{nicefrac}       
\usepackage{microtype}      
\usepackage[linesnumbered,boxed,lined,ruled,vlined]{algorithm2e}
\usepackage{amsmath, amssymb, graphicx, url}
\usepackage{paralist}
\usepackage{booktabs}
\usepackage{multirow}
\usepackage[utf8]{inputenc} 
\usepackage[T1]{fontenc}    
\usepackage{hyperref}       
\usepackage{url}            
\usepackage{booktabs}       
\usepackage{amsfonts}       
\usepackage{nicefrac}       
\usepackage{microtype}      
\usepackage{amsmath, amsthm, amssymb, multirow, paralist}
\usepackage{enumitem}
\usepackage{color}

\newtheorem{theorem}{Theorem}
\newtheorem{prop}{Proposition}
\newtheorem{lemma}{Lemma}

\newtheorem{ass}{Assumption}

\def \S {\mathbf{S}}

\def \R {\mathbb{R}}

\def \v {\mathbf{v}}

\def \x {\mathbf{x}}

\def \E {\mathrm{E}}

\def \x {\mathbf{x}}

\def \y {\mathbf{y}}

\def \g {\mathbf{g}}

\def \y {\mathbf{y}}
\def \E {\mathbb{E}}
\def \x {\mathbf{x}}
\def \g {\mathbf{g}}

\def \R {\mathbb{R}}
\def \S {\mathcal{S}}

\def \v {\mathbf{v}}

\title{Stochastic Non-convex Optimization with  Strong High Probability Second-order Convergence}

\author{
 Mingrui Liu, Tianbao Yang \\
  Department of Computer Science\\
  The University of Iowa, Iowa City, IA 52242 \\
  \texttt{mingrui-liu, tianbao-yang@uiowa.edu} \\
}

\begin{document}

\maketitle

\begin{abstract}
In this paper, we study stochastic non-convex  optimization with non-convex random functions. 
Recent studies on non-convex optimization revolve around establishing second-order convergence, i.e., converging to a nearly second-order optimal stationary points. However, existing results on stochastic non-convex optimization are limited, especially with a high probability second-order convergence. 
We propose a novel updating step (named NCG-S) by leveraging  a  stochastic gradient and a noisy negative curvature of a stochastic Hessian, where the stochastic gradient and Hessian are based on a proper mini-batch of random functions. Building on this step, we develop two algorithms and establish their high probability second-order convergence. 
To the best of our knowledge, the proposed stochastic algorithms are the first with a second-order convergence in {\it high probability} and a time complexity that is {\it almost linear} in the problem's dimensionality.
\end{abstract}

\section{Introduction}
\vspace*{-0.1in}
In this paper, we consider the following stochastic optimization problem:
\begin{align}\label{eqn:stocn}
\min_{\x\in\R^d}f(\x)=\E_{\xi}[f(\x; \xi)],
\end{align}
where $f(\x; \xi)$ is a random function but not necessarily convex. 
 The above formulation plays an important role for solving many machine learning problems, e.g.,  deep learning~\cite{goodfellow2016deep}.

A prevalent algorithm for solving the problem is stochastic gradient descent (SGD)~\cite{ghadimi2013stochastic}. However, SGD can only guarantee convergence  to a first-order stationary point (i.e., $\|\nabla f(\x)\|\leq\epsilon_1$, where $\|\cdot\|$ denotes the Euclidean norm) for non-convex optimization, which could be a saddle point. 
A potential solution to address this issue is to find a nearly second-order stationary point $\x$ such that $\|\nabla f(\x)\|\leq \epsilon_1\ll 1$, and $-\lambda_{\text{min}}(\nabla^2 f(\x))\leq \epsilon_2\ll 1$, where $\lambda_{\text{min}}(\cdot)$ denotes the smallest eigenvalue. When the objective function is non-degenerate (e.g., strict saddle~\cite{pmlr-v40-Ge15} or whose Hessian at all saddle points has a negative eigenvalue), an approximate second-order stationary point is  close to a local minimum. 

Although there emerged a number of algorithms for finding a nearly second-order stationary point for  non-convex optimization with a deterministic function~\cite{nesterov2006cubic,conn2000trust,Cartis2011,Cartis2011b,DBLP:conf/stoc/AgarwalZBHM17,DBLP:journals/corr/CarmonDHS16,royer2017complexity}, results for stochastic non-convex optimization are still limited. There are three closely related works~\cite{pmlr-v40-Ge15, DBLP:conf/colt/ZhangLC17,natasha2}. A summary of algorithms in these works and their convergence results is presented in Table~\ref{tab:2}. It is notable that Natasha2, which involves switch between several sub-routines including SGD, a degenerate version of Natasha1.5 for finding a first-order stationary point, and an online power method (i.e., the Oja's algorithm~\cite{oja1982simplified}) for computing  the negative curvature (i.e., the eigen-vector corresponding to the minium eigen-value) of the Hessian matrix,  is more complex than noisy SGD and SGLD.

\begin{table*}[t]
		\caption{Comparison with existing stochastic algorithms for achieving an $(\epsilon_1, \epsilon_2)$-second-order stationary solution to~(\ref{eqn:stocn}), where $p$ is a number at least $4$, IFO (incremental first-order oracle) and ISO (incremental second-order oracle) are terminologies borrowed from~\cite{reddi2017generic}, representing $\nabla f(\x; \xi)$ and  $\nabla^2 f(\x; \xi)\v$ respectively,  $T_h$ denotes the runtime of ISO and $T_g$ denotes the runtime of IFO. The proposed algorithms SNCG have two variants with different time complexities, where the result marked with $*$ has a practical improvement detailed later. } 
		\centering
		\label{tab:2}
		\begin{small}\begin{tabular}{l|lll}
			\toprule
			algo.& oracle & second-order guarantee in &time complexity\\
			&& expectation or high probability&\\
			\midrule
			Noisy SGD~\cite{pmlr-v40-Ge15} &IFO&$(\epsilon, \epsilon^{1/4})$, high probability&$\widetilde O\left(T_gd^p\epsilon^{-4}\right)$\\ 
			\midrule
			SGLD~\cite{DBLP:conf/colt/ZhangLC17} &IFO&$(\epsilon, \epsilon^{1/2})$, high probability&$\widetilde O\left(T_gd^p\epsilon^{-4}\right)$\\ 

			\midrule
			Natasha2~\cite{natasha2} &IFO  + ISO&$(\epsilon, \epsilon^{1/2})$, expectation&$\widetilde O\left( T_g\epsilon^{-3.5}+T_h\epsilon^{-2.5} \right)$\\ 
			 \midrule
            SNCG&IFO + ISO&$(\epsilon, \epsilon^{1/2})$, high probability&$\widetilde O\left(T_g\epsilon^{-4} + T_h\epsilon^{-3}\right)^*$\\
                        &&&$\widetilde O\left(T_g\epsilon^{-4} + T_h\epsilon^{-2.5}\right)$\\
			\bottomrule
		\end{tabular}
		\end{small}
		\vspace*{-0.2in}
	\end{table*}

In this paper, we propose new stochastic optimization algorithms for solving~(\ref{eqn:stocn}). Similar to several existing algorithms, we also use the negative curvature  to escape from saddle points. The key difference is that we compute a noisy negative curvature based on a proper mini-batch of sampled random functions. A novel updating step is proposed that follows a  stochastic gradient or the noisy negative curvature depending on which decreases the objective value most.  Building on this step, we present two algorithms that have different time complexities. A summary of our results and comparison with previous similar results are presented in Table~\ref{tab:2}. To the best of our knowledge, the proposed algorithms are the first for stochastic non-convex optimization with a second-order convergence in {\it high probability} and a time complexity that is {\it almost linear} in the problem's dimensionality. It is also notable that our result is much stronger than the mini-batch SGD analyzed in~\cite{Ghadimi:2016:MSA:2874819.2874863} for stochastic non-convex optimization in that (i) we use the same number of IFO as in~\cite{Ghadimi:2016:MSA:2874819.2874863} but achieve the second-order convergence using a marginal number of ISO; (ii) our high probability convergence is for a solution from a single run of the proposed algorithms instead of from multiple runs and using a boosting technique as in~\cite{Ghadimi:2016:MSA:2874819.2874863}. 

Before moving to the next section, we would like to remark that  stochastic algorithms with second-order convergence result are recently proposed for solving a finite-sum problem~\cite{reddi2017generic}, which alternates between a first-order sub-routine (e.g., stochastic variance reduced gradient) and a second-order sub-routine (e.g., Hessian descent). Since full gradients are computed  occasionally, they are not applicable to the general stochastic non-convex optimization problem~(\ref{eqn:stocn}) and hence are excluded from comparison. Nevertheless, our idea of the proposed NCG-S step that lets negative curvature descent competes with the gradient descent can be borrowed to reduce the number of stochastic Hessian-vector products in their Hessian descent. We will elaborate this point later. 


\section{Preliminaries and Building Blocks}
\vspace*{-0.1in}
Our goal is to find an $(\epsilon_1, \epsilon_2)$-second order stationary point $\x$ such that
 $\|\nabla f(\x)\|\leq \epsilon_1$, and $\lambda_{\min}(\nabla^2 f(\x))\geq -\epsilon_2$.
To this end, we make the following assumptions regarding~(\ref{eqn:stocn}).
\begin{ass}\label{ass:1} (i) Every random function $f(\x; \xi)$ is twice differentiable, and it has Lipschitz continuous gradient, i.e., there exists $L_1>0$ such that  $\|\nabla f(\x; \xi) - \nabla f(\y; \xi)\|\leq L_1\|\x - \y\|$,  (ii) $f(\x)$  has Lipschitz continuous Hessian, i.e.,  there exists $L_2>0$ such that $\|\nabla^2 f(\x) - \nabla^2 f(\y)\|_2\leq L_2\|\x - \y\|$, (iii)  given an initial point $\x_0$, there exists $\Delta<\infty$ such that $f(\x_0) - f(\x_*)\leq \Delta$, where $\x_*$ denotes the global minimum of $f(\x)$; (iv) there exists $G>0$ such that $\mathbb{E}[\exp(\|\nabla f(\x; \xi) - \nabla f(\x)\|/G)]\leq \exp(1)$ holds. 
\end{ass}
\vspace*{-0.1in}
{\bf Remark:} The first three assumptions are standard assumptions for non-convex optimization in order to establish second-order convergence. The last assumption is standard for stochastic optimization necessary for high probability analysis.

The proposed algorithms require noisy first-order information at each iteration and maybe noisy second-order information. We first discuss approaches to compute these information, which will lead us to the updating step NCG-S. To compute noisy first-order information, we use incremental first-order oracle (IFO) that takes $\x$ as input and returns $\nabla f(\x; \xi)$. In particular, at a point $\x$ we sample a set of random variables $\S_1 = \{\xi_1, \xi_2, \ldots,\}$ and compute  a stochastic  gradient $\g(\x) = \frac{1}{|\mathcal S_1|}\sum_{\xi_i\in\S_1}\nabla f(\x; \xi_i)$ such that $\|\g(\x) - \nabla f(\x)\|\leq  \epsilon_4\leq \min(\frac{1}{2\sqrt{2}}\epsilon_1, \epsilon_2^2/(24L_2))$ holds with high probability. This can be guaranteed by the following lemma.
\begin{lemma}
	\label{lem:gc}
	Suppose {\bf Assumption 1} (iv) holds. Let $\g(\x) = \frac{1}{|\mathcal S_1|}\sum_{\xi_i\in\S_1}\nabla f(\x; \xi_i)$.  For any $\epsilon_4,\delta\in(0,1)$, $\x\in\R^d$, when 
	$|\mathcal{S}_1|\geq\frac{4G^2(1+3\log^2(1/\delta))}{\epsilon_4^2}$,
	we have 
$	\Pr(\|\g(\x)-\nabla  f(\x)\|\leq\epsilon_4)\geq 1-\delta.$
\end{lemma}
The lemma can be proved by using large deviation theorem of vector-valued martingales (e.g., see~\cite{Ghadimi:2016:MSA:2874819.2874863}[Lemma 4]). 

To compute noisy second-order information, we calculate a noisy negative curvature of a stochastic Hessian that is sufficiently close to the true Hessian. In particular, at a point $\x$ we sample a set of random variables $\S_2 = \{\xi'_1, \xi'_2, \ldots, \}$ and compute  a noisy negative curvature $\v$ of the stochastic Hessian $H(\x) = \frac{1}{|\mathcal S_2|}\sum_{\xi'_i\in\S_2}\nabla^2 f(\x; \xi'_i)$, where $|\S_2|$ is sufficiently large such that $\|H(\x) - \nabla^2 f(\x)\|_2\leq \epsilon_3\leq \epsilon_2/24$ holds with high probability, where $\|\cdot\|_2$ denotes the spectral norm of a matrix. This can be guaranteed according to the following lemma. 
\begin{lemma}
	\label{lem:Hc}
	 Suppose {\bf Assumption 1} (i) holds. Let $H(\x) = \frac{1}{|\mathcal S_2|}\sum_{\xi_i\in\S_2}\nabla^2 f(\x; \xi_i)$. For any $\epsilon_3,\delta\in(0,1)$, $\x\in\R^d$, when $|\mathcal{S}_2|\geq\frac{16L_1^2}{\epsilon_3^2}\log(\frac{2d}{\delta})$, we have 
$	\Pr(\|H(\x)-\nabla^2 f(\x)\|_2\leq\epsilon_3)\geq 1-\delta'.$
\end{lemma}
The above lemma can be proved by using matrix concentration inequalities. Please see \cite{peng16inexacthessian}[Lemma 4] for a proof.  To compute a noisy negative curvature of $H(\x)$, we can leverage approximate PCA algorithms~\cite{DBLP:conf/nips/ZhuL16,DBLP:conf/icml/GarberHJKMNS16} using the incremental second-order oracle (ISO) that can compute $\nabla^2 f(\x; \xi)\v$. 
\begin{lemma}\label{lem:approxPCA}
	Let $H = \frac{1}{m}\sum_{i=1}^mH_i$ where $\|H_i\|_2\leq L_1$. There exists a randomized algorithm $\mathcal A$ such that with probability at least $1- \delta$, $\mathcal A$ produces a unit vector $\v$ satisfying  $\lambda_{\min}(H)\geq \v^{\top}H\v - \varepsilon$ with a time complexity of $\widetilde O(T_h^1\max\{m, m^{3/4}\sqrt{L_1/\varepsilon}\})$, where $T_h$ denotes the time of computing $H_i\v$ and $\widetilde O$ suppresses a logarithmic term in $d, 1/\delta,  1/\varepsilon$. 
\end{lemma}

\textbf{NCG-S: the updating step.} With the approaches for computing noisy first-order and second-order information, we present  a novel updating step called NCG-S in Algorithm \ref{alg:sgnc}, which uses a competing idea that takes a step along the noisy negative gradient direction or the  noisy negative curvature direction depending on which decreases the objective value more. One striking feature of NCG-S is that the noise level in computing a noisy negative curvature of $H(\x)$ is set to a free parameter $\varepsilon$ instead of the target accuracy level $\epsilon_2$ as in many previous works~\cite{DBLP:conf/stoc/AgarwalZBHM17,DBLP:journals/corr/CarmonDHS16,peng16inexacthessian}, which allows us to design an algorithm with a much reduced number of ISO calls in practice. The following lemma justifies the fact of sufficient decrease in terms of the objective value of each NCG-S step.
\begin{lemma}
	\label{lemma:ncg-s}
		Suppose Assumption 1 holds.
		Conditioned on the event $\mathcal A=\{\|H(\x_j) - \nabla^2 f(\x_j)\|_2\leq \epsilon_3\} \cap \{\|\g(\x_j) - \nabla f(\x_j)\|\leq \epsilon_4\}$ where $\epsilon_3\leq \epsilon_2/24$ and $\epsilon_4 \leq \min(\frac{1}{2\sqrt{2}}\epsilon_1, \epsilon_2^2/(24L_2))$, the update $\x_{j+1}=\text{NCG-S}(\x_j,\varepsilon,\delta,\epsilon_2)$ satisfies
		$	f(\x_j) - f(\x_{j+1})\geq \max\left(\frac{1}{4L_1}\|\g(\x_j)\|^2   -  \frac{\epsilon_1^2}{8L_1}, \frac{-\epsilon_2^2\v_j^\top H(\x_j)\v_j}{2L_2^2} - \frac{11\epsilon_2^3}{48L_2^2}\right).$
\end{lemma}
 \setlength\floatsep{0.1\baselineskip plus 3pt minus 1pt}
\setlength\textfloatsep{0.1\baselineskip plus 1pt minus 1pt}
\setlength\intextsep{0.1\baselineskip plus 1pt minus 1 pt}
\begin{algorithm}[t]
	\caption{The stochastic NCG step: $(\x^+, \v^\top H(\x)\v)=\text{NCG-S}(\x, \varepsilon, \delta,\epsilon_1, \epsilon_2)$}\label{alg:sgnc}
		\textbf{Input}:  $\x$, $\varepsilon$, $\delta$, $\epsilon_1, \epsilon_2$\;
		let  $\g(\x)$ and $H(\x)$ be a stochastic gradient and Hessian according to Lemma~\ref{lem:gc} and~\ref{lem:Hc}\;
		Find a unit vector $\v$ such that $\lambda_{\min}(H(\x))\geq \v^{\top}H(\x)\v - \varepsilon$ 
		according to  Lemma~\ref{lem:approxPCA}\;
		\If{$-\frac{\epsilon_2^2}{2L_2^2}\v^\top H(\x)\v-\frac{11\epsilon_2^3}{48L_2^2}>\frac{\|\g(\x)\|^2}{4L_1} -  \frac{\epsilon_1^2}{8L_1}$}{
		Compute $\x^+  = \x - \frac{\epsilon_2}{L_2}\text{sign}(\v^{\top}\g(\x))\v$\;
		}
		\Else{
		Compute $\x^+ = \x - \frac{1}{L_1}\g(\x)$\;
		
		}
	    return $\x^+, \v^\top H(\x)\v$
\end{algorithm}
\section{The Proposed Algorithms: SNCG}
\vspace*{-0.1in}
In this section, we present two variants of the proposed algorithms based on the NCG-S step shown in Algorithm~\ref{alg:sgncA} and Algorithm~\ref{alg:SGSNCG}.  The differences of these two variants are (i) SNCG-1 uses NCG-S at every iteration to update the solution, while SNCG-2 only uses NCG-S when the approximate gradient's norm is small; (ii) the noise level $\varepsilon$ for computing the noisy negative curvature (as in Lemma~\ref{lem:approxPCA}) in SNCG-1 is set to $\max(\epsilon_2, \|\g(\x_j)\|^\alpha)/2$ adaptive to the magnitude of the stochastic gradient, where $\alpha\in(0,1]$ is a parameter that characterizes $\epsilon_2 = \epsilon_1^\alpha$. In contrast, the noise level $\varepsilon$ in SNCG-2 is simply set to $\epsilon_2/2$. These differences lead to different time complexities of the two algorithms. 


 \setlength\floatsep{0.1\baselineskip plus 3pt minus 2pt}
\setlength\textfloatsep{0.1\baselineskip plus 1pt minus 2pt}
\setlength\intextsep{0.1\baselineskip plus 1pt minus 2 pt}
\begin{algorithm}[t]
\DontPrintSemicolon
	\caption{SNCG-1: $(\x_0, \epsilon_1, \alpha, \delta)$}\label{alg:sgncA}
	\textbf{Input}:  $\x_0$, $\epsilon_1, \alpha$, $\delta$\;
	Set $\x_1=\x_0$, $\epsilon_2 = \epsilon_1^\alpha$, $\delta' = \delta /(1+\max\left(\frac{48L_2^2}{\epsilon_2^3}, \frac{8L_1}{\epsilon_1^2}\right)\Delta)$\; 
		\For{$j=1,2,\ldots,$}{
		$(\x_{j+1}, \v_j^\top H(\x_j)\v_j) = \text{NCG-S}(\x_j,  \max(\epsilon_2, \|\g(\x_j)\|^\alpha)/2,  \delta', \epsilon_1, \epsilon_2)$\;
		
		\If{ $\v_j^{\top}H(\x_j)\v_j> -\epsilon_2/2$ and $\|\g(\x_j)\|\leq \epsilon_1$}
		{return $\x_j$}
		}
\end{algorithm}

\begin{algorithm}[t]
\DontPrintSemicolon
	\caption{SNCG-2: $(\x_0, \epsilon_1, \delta)$}\label{alg:SGSNCG}
		\textbf{Input}:  $\x_0$, $\epsilon_1,  \delta$\;
		Set $\x_1=\x_0$,  $\delta' = \delta /(1+\max\left(\frac{48L_2^2}{\epsilon_2^3}, \frac{8L_1}{\epsilon_1^2}\right)\Delta)$\;
		\For{$j=1,2,\ldots,$}{
		Compute $\g(\x_j)$ according to Lemma~\ref{lem:gc}\;
		\If{$\|\g(\x_j)\|\geq\epsilon_1$}{
	        compute $\x_{j+1}=\x_j-\frac{1}{L_1}\g(\x_j)$// SG step\;		}
		\Else{
		compute $(\x_{j+1},\v_j^\top H(\x_j)\v_j)=\text{NCG-S}(\x_j,\epsilon_2/2,\delta',\epsilon_1, \epsilon_2)$\;
		\If{$\v_j^\top H(\x_j)\v_j>-\epsilon_2/2$}
		{
		 return $\x_j$\;
		}
		}
		}
\end{algorithm}

\begin{theorem}
	\label{cor:SGSNCG}
	Suppose Assumption~\ref{ass:1} holds, 
	 $\epsilon_3\leq \epsilon_2/24$ and $\epsilon_4 \leq \min(\frac{1}{2\sqrt{2}}\epsilon_1, \epsilon_2^2/(24L_2))$.
	With probability $1-\delta$, SNCG-1 terminates with at most $[1+\max\left(\frac{48L_2^2}{\epsilon_2^3},\frac{8L_1}{\epsilon_1^2}\right)\Delta]$ NCG-S steps, 
	and  furthermore, each NCG-S step requires time in the order of $
	\widetilde O\left(T_h|\mathcal S_2| + T_h|\mathcal S_2|^{3/4} \frac{\sqrt{L_1}}{\max(\epsilon_2, \|\g(\x_j)\|^\alpha)^{1/2}} + |\mathcal S_1|T_g\right)$; 
SNCG-2 terminates with at most $\frac{8L_1}{\epsilon_1^2}\Delta$ SG steps and at most $(1+\frac{48L_2^3}{\epsilon_2^3})\Delta$ NSG-S steps,  each NCG-S step requires time in the order of $
	\widetilde O\left(T_h|\mathcal S_2| + T_h|\mathcal S_2|^{3/4} \frac{\sqrt{L_1}}{\epsilon_2^{1/2}} + |\mathcal S_1|T_g\right)$.
	Upon termination, with probability $1-3\delta$, both  algorithms return a solution $\x_{j_*}$ such that $\|\nabla f(\x_{j_*})\|\leq2\epsilon_1 $ and $\lambda_{\text{min}}\left(\nabla^2f(\x_{j_*})\right)\geq -2\epsilon_2.$ 
%

\end{theorem}
\vspace*{-0.1in}
{\bf Remark:} To analyze the time complexity, we can plug in the order of $|\S_1|$ and $|\S_2|$ as in Lemma~\ref{lem:gc} and Lemma~\ref{lem:Hc}. It is not difficult to show that when $\epsilon_2=\sqrt{\epsilon_1}$, the worst-case time complexities of these two algorithms are given in Table~\ref{tab:2}, where the result marked by $^*$ corresponds to SNCG-1. However, this worse-case result is computed by simply bounding $T_h/\sqrt{\max(\epsilon_2, \|\g(\x)\|^\alpha)}$ by $T_h/\sqrt{\epsilon_2}$. In practice, before reaching a saddle point (i.e., $\|\g(\x_j)\|\geq \epsilon_1)$, the number of ISO calls for each NCG-S step in SNCG-1 can be less than that of each NCG-S step in SNCG-2. In addition, the NCG-S step in SNCG-1 can be faster than the SG step in SNCG-2 before reaching a saddle point.  More importantly, the idea of competing between gradient descent and negative curvature descent and the adaptive noise parameter $\varepsilon$ for computing the noisy negative curvature can be also useful in other algorithms. For example, in~\cite{reddi2017generic} the Hessian descent (also known as negative curvature descent) can take the competing idea and uses adaptive noise level for computing a noisy negative curvature. 


\section{Conclusion}
\vspace*{-0.1in}
In this paper, we have proposed  new algorithms for stochastic  non-convex optimization with strong high probability second-order convergence guarantee. To the best of our knowledge, the proposed stochastic algorithms are the first with a second-order convergence in {high probability} and a time complexity that is almost linear in the problem's dimensionality.
{
\bibliographystyle{abbrv}
\bibliography{ref}}
\newpage
\appendix
\section{Proof of Lemma~\ref{lem:approxPCA}}
We first introduce a proposition, which is the Theorem 2.5 in~\cite{DBLP:conf/stoc/AgarwalZBHM17}.
\begin{prop}
	\label{agarwal:thm2.5}
	Let $M\in\R^{d\times d}$ be a symmetric matrix with eigenvalues $1\geq\lambda_1\ldots\geq\lambda_d\geq 0$. Then with probability at least $1-p$, the Algorithm AppxPCA produces a unit vector $\v$ such that $\v^\top M\v\geq(1-\delta_+)(1-\epsilon)\lambda_{\text{max}}(M)$. The total running time is $\widetilde O\left(T_h^1\max\{m,\frac{m^{3/4}}{\sqrt{\epsilon}}\}\log^2\left(\frac{1}{\epsilon^2\delta_+}\right)\right)$.
\end{prop}
\begin{proof}[Proof of Lemma~\ref{lem:approxPCA}]
	Define $M=I-\frac{H}{L_1}$, then $M$ satisfies the condition in the Proposition \ref{agarwal:thm2.5}. Then we know that with probability at least $1- p$, the Algorithm AppxPCA produces a vector $\v$ satisfying
	\begin{align*}
	\v^\top\left(I-\frac{H}{L_1}\right)\v\geq(1-\delta_+)(1-\epsilon)\left(1-\frac{\lambda_{\text{min}}(H)}{L_1}\right),
	\end{align*}
	which implies that
	\begin{align*}
	L_1-\v^\top H\v\geq (1-\delta_+ -\epsilon + \delta_+\epsilon)(L_1-\lambda_{\text{min}}(H))\geq (1-\delta_+ -\epsilon)(L_1-\lambda_{\text{min}}(H)).
	\end{align*}
	By simple algebra, we have
	\begin{align*}
	\lambda_{\text{min}}(H)&\geq\v^\top H\v-(\delta_++\epsilon)(L_1-\lambda_{\text{min}}(H))\\
	&\geq\v^\top H\v-2L_1(\delta_++\epsilon).
	\end{align*}
	By setting $\epsilon=\delta_+=\frac{\varepsilon}{4L_1}$, we can finish the proof.  
\end{proof}
\section{Proof of Lemma \ref{lemma:ncg-s}}
\begin{proof}
	Define $\eta_j=\frac{\epsilon_2}{L_2}\text{sign}(\v_j^\top\g(\x_j))$. Next, we analyze the objective decrease for $j$-th  NCG-S step conditioned on the event $\mathcal A=\{\|H(\x_j) - \nabla^2 f(\x_j)\|_2\leq \epsilon_3 \cap \|\g(\x_j) - \nabla f(\x_j)\|\leq \epsilon_4\}$ and let $\text{Pr}(\mathcal A)=1-\delta'$.
	
	By $L_2$-Lipschitz continuity of Hessian, 
	we know that
	\begin{align*}
	f(\x_{j+1}^1)-f(\x_j)\leq -\eta_j\nabla f(\x_j)^{\top}\v_j + \frac{\eta_j^2}{2}\v_j^\top \left(\nabla^2 f(\x_j)-H(\x_j)\right)\v_j+\frac{\eta_j^2}{2}\v_j^\top H(\x_j)\v_j+\frac{L_2}{6}|\eta_j|^3 .
	\end{align*}
	where $\x_{j+1}^1$ is an update of $\x_j$ following $\v_j$ in NCG-S. 
	Note that $\epsilon_4\leq \epsilon_2^2/24L_2$, and then we have
	\begin{align}
	& -\eta_j\nabla f(\x_j)^{\top}\v_j =  -\eta_j \g(\x_j)^{\top}\v_j + \eta_j(\g(\x_j) - \nabla f(\x_j))^{\top}\v_j\leq |\eta_j\epsilon_4| \leq \frac{\epsilon_2^3}{24L_2^2}\\
	& \v_j^\top \left(\nabla^2 f(\x_j)-H(\x_j)\right)\v_j\leq \epsilon_3\leq \epsilon_2/24
	\end{align}
	Then it follows that
	\begin{align}
	f(\x^1_{j+1})-f(\x_j)\leq \frac{\epsilon_2^3}{24L_2^2}+\frac{\epsilon_2^3}{48L_2^2}+\frac{\epsilon_2^2\v_j^\top H(\x_j)\v_j}{2L_2^2}+\frac{\epsilon_2^3}{6L_2^2}=-\underbrace{\left(\frac{-\epsilon_2^2\v_j^\top H(\x_j)\v_j}{2L_2^2} - \frac{11\epsilon_2^3}{48L_2^2}\right)}\limits_{\Delta_1}.
	\end{align}
	Similarly, let $\x^2_{j+1}$ denote an update of $\x_j$ following $\g(\x_j)$ in NCG-S,  we have
	\begin{align*}
	f(\x^2_{j+1})-f(\x_j)&\leq (\x^2_{j+1} - \x_j)^{\top}\nabla f(\x_j) + \frac{L_1}{2}\|\x^2_{j+1} - \x_j\|^2\\
	&= -\frac{1}{L_1}\g(\x_j)^{\top}\nabla f(\x_j) + \frac{\|\g(\x_j)\|^2}{2L_1}\\
	&=-\frac{1}{L_1}\g(\x_j)^{\top}\g(\x_j) + \frac{1}{L_1}\g(\x_j)^{\top}(\g(\x_j) - \nabla f(\x_j)) +  \frac{\|\g(\x_j)\|^2}{2L_1}\\
	&\leq - \frac{1}{2L_1}\|\g(\x_j)\|^2  + \frac{1}{4L_1}\|\g(\x_j)\|^2 + \frac{1}{L_1}\|\g(\x_j) - \nabla f(\x_j)\|^2\\
	&= - \frac{1}{4L_1}\|\g(\x_j)\|^2  + \frac{1}{L_1}\epsilon_4^2\leq - \underbrace{\frac{1}{4L_1}\|\g(\x_j)\|^2  + \frac{\epsilon_1^2}{8L_1}}\limits_{-\Delta_2}\\
	\end{align*}
	where we use $\epsilon_4\leq \frac{1}{2\sqrt{2}}\epsilon_1$.
	According to the update of NCG-S, if $\Delta_1>\Delta_2$, we have $\x_{j+1} = \x^1_{j+1}$ and then $f(\x_{j}) - f(\x_{j+1})\geq \Delta_1 = \max(\Delta_1, \Delta_2)$. If $\Delta_2\geq \Delta_1$, we have $\x_{j+1} = \x^2_{j+1}$ and then $f(\x_{j}) - f(\x_{j+1})\geq \Delta_2 = \max(\Delta_1, \Delta_2)$. Therefore, with probability $1- \delta'$ we have, 
	\begin{align*}
	f(\x_j) - f(\x_{j+1})\geq \max\left(\frac{1}{4L_1}\|\g(\x_j)\|^2   -  \frac{\epsilon_1^2}{8L_1}, \frac{-\epsilon_2^2\v_j^\top H(\x_j)\v_j}{2L_2^2} - \frac{11\epsilon_2^3}{48L_2^2}\right)
	\end{align*}
	\end{proof}
\section{Proof of Theorem \ref{cor:SGSNCG}}
\begin{proof}
\begin{itemize}
\item We first prove the result of SNCG-1.
	For the $j$-th  NCG-S step, define the event $\mathcal A=\{\|H(\x_j) - \nabla^2 f(\x_j)\|_2\leq \epsilon_3\} \cap \{\|\g(\x_j) - \nabla f(\x_j)\|\leq \epsilon_4\}$ and let $\text{Pr}(\mathcal A)=1-\delta'$ (we can choose $\epsilon_3$ and $\epsilon_4$ to make it hold). Since the Algorithm SCNG-1 calls NCG-S as a subroutine, then by Lemma \ref{lemma:ncg-s}, we know that with probability at least $1-\delta'$,
	\begin{align*}
	f(\x_j) - f(\x_{j+1})\geq \max\left(\frac{1}{4L_1}\|\g(\x_j)\|^2   -  \frac{\epsilon_1^2}{8L_1}, \frac{-\epsilon_2^2\v_j^\top H(\x_j)\v_j}{2L_2^2} - \frac{11\epsilon_2^3}{48L_2^2}\right).
	\end{align*}
	If $\v_j^\top H(\x_j)\v_j\leq -\epsilon_2/2$, we have $\Delta_1\geq \frac{\epsilon_2^3}{48L_2^2}$ and
	\begin{align*}
	f(\x_j) - f(\x_{j+1})\geq  \frac{\epsilon_2^3}{48L_2^2}
	\end{align*}
	If $\|\g(\x_j)\|> \epsilon_1$, we have $\Delta_2\geq \frac{\epsilon_1^2}{8L_1}$ and 
	\begin{align*}
	f(\x_j) - f(\x_{j+1})\geq  \frac{\epsilon_1^2}{8L_1}
	\end{align*}
	Therefore, before the algorithm terminates, i.e., for all iterations $j\leq j_* -1$, we have either $\v_j^\top H(\x_j)\v_j\leq -\epsilon_2/2$ or $\|\g(\x_j)\|> \epsilon_1$. In either case, the following holds with probability $1-\delta'$
	\begin{align*}
	f(\x_j) - f(\x_{j+1})\geq  \min\left(\frac{\epsilon_1^2}{8L_1}, \frac{\epsilon_2^3}{48L_2^2}\right),
	\end{align*}
	from which we can derive the upper bound of $j_*$, which is $j_*\leq [1+\max\left(\frac{48L_2^2}{\epsilon_2^3},\frac{8L_1}{\epsilon_1^2}\right)\Delta]$. Next, we show that upon termination, we achieve an $(2\epsilon_1, 2\epsilon_2)$-second order stationary point with high probability. In particular, with probability $1-\delta'$ we have
	\[
	\|\nabla f(\x_{j_*})\|\leq \|\nabla f(\x_{j_*}) - \g(\x_{j_*})\| + \|\g(\x_{j_*})\| \leq \epsilon_4  + \epsilon_1 \leq 2\epsilon_1. 
	\]
	and  with probability $1-\delta'$
	\begin{align*}
	\lambda_{\min}(H(\x_{j_*}))\geq \v_{j_*}^{\top}H(\x_{j_*})\v_{j_*}  -\max\left(\epsilon_2,\|g(\x_{j_*})\|^\alpha\right)/2 \geq -\epsilon_2 
	\end{align*}
	In addition, with probability $1-\delta'$, we have
	\[
	\lambda_{\min}(\nabla^2 f(\x_{j_*}))\geq \lambda_{\min}(H(\x_{j_*})) - \epsilon_3\geq-2\epsilon_2 
	\]
	As a result, by using union bound, we have with probability $1-3j_*\delta'=1-3\delta$, we have
	\begin{align*}
	\|\nabla f(\x_{j_*})\|\leq2\epsilon_1, \quad   \lambda_{\min}(\nabla^2 f(\x_{j_*}))\geq -2\epsilon_2
	\end{align*}
	Finally, the time complexity of each iteration follows Lemma~\ref{lem:approxPCA}. 

\item The proof of the result of SNCG-2 is similar. For simplicity, we use the same notation unless specified.
	According to the Algorithm SNCG-2, we know that when $\|\g(\x_j)\|\geq\epsilon_1$, the SG step guarantees that 
	\begin{align*}
	f(\x_{j+1})-f(\x_j)\leq-\frac{1}{4L_1}\|\g(\x_j)\|^2+\frac{\epsilon_1^2}{8L_1}\leq -\frac{\epsilon_1^2}{8L_1}
	\end{align*}
	When $\v_j^\top H(\x)\v_j\leq-\epsilon_2/2$, then the NCG-S step guarantees that 
	\begin{align*}
	f(\x_{j+1})-f(\x_j)\leq -\max\left(\frac{\epsilon_2^2}{48L_2^2},\frac{1}{4L_1}\|\g(\x_j)\|^2   -  \frac{\epsilon_1^2}{8L_1}\right)\leq -\frac{\epsilon_2^3}{48L_2^2}
	\end{align*}
According to the update rule, it is easy to see that conditioned on the event $\mathcal{A}$, the SG step always decrease the objective value (with high probability) and the NCG-S step decreases the objective value at all steps except for the very last iteration (with high probability). Denote $j_*$ by the number of iterations in the Algorithm SNCG-2. By the sufficient decrease argument, we know that with probability at least $1-j_*\delta'$, the algorithm terminates, where 
\begin{align*}
	j_*\leq 1+\max\left(\frac{48L_2^2}{\epsilon_2^3},\frac{8L_1}{\epsilon_1^2}\right).
\end{align*}
By the relationship between $\delta'$ and $\delta$, we know that the algorithm terminates with probability at least $1-\delta$.

	Note that $f(\x_0)-f(\x_*)\leq\Delta$, and hence with probability $1-j_*\delta'$,
we have at most $\frac{8L_1}{\epsilon_1^2}\Delta$ stochastic gradient evaluations and at most $(1+\frac{48L_2^3}{\epsilon_2^3})\Delta$ stochastic Hessian-vector product evaluations before termination.

Next, we show that upon termination, we achieve an $(2\epsilon_1, 2\epsilon_2)$-second order stationary point with high probability. In particular, with probability $1-\delta'$ we have
	\[
	\|\nabla f(\x_{j_*})\|\leq \|\nabla f(\x_{j_*}) - \g(\x_{j_*})\| + \|\g(\x_{j_*})\| \leq \epsilon_4  + \epsilon_1 \leq 2\epsilon_1. 
	\]
	and  with probability $1-\delta'$
	\begin{align*}
	\lambda_{\min}(H(\x_{j_*}))\geq \v_{j_*}^{\top}H(\x_{j_*})\v_{j_*} - \epsilon_2/2 \geq -\epsilon_2 
	\end{align*}
	In addition, with probability $1-\delta'$, we have
	\[
	\lambda_{\min}(\nabla^2 f(\x_{j_*}))\geq \lambda_{\min}(H(\x_{j_*})) - \epsilon_3 \geq -2\epsilon_2
	\]
	As a result, by using union bound, we have with probability $1-3j_*\delta'=1-3\delta$, we have
	\begin{align*}
	\|\nabla f(\x_{j_*})\|\leq2\epsilon_1, \quad   \lambda_{\min}(\nabla^2 f(\x_{j_*}))\geq -2\epsilon_2
	\end{align*}
	Finally, the time complexity of each iteration follows Lemma~\ref{lem:approxPCA}. 
\end{itemize}
\end{proof}
\end{document}